\documentclass[12 pt]{amsart}

\usepackage{amsthm}
\usepackage{amsmath}
\usepackage{amssymb}
\usepackage{latexsym}
\usepackage{url}
\usepackage{graphicx}
\usepackage{bmpsize}

\author{Felix Goldberg}
\address{Caesarea-Rothschild Institute, University of Haifa, Haifa, Israel}
\email{felix.goldberg@gmail.com}

\title{New results on eigenvalues and degree deviation}
\date{February 11, 2013}

\newtheorem{thm}{Theorem}[section]
\newtheorem{cor}[thm]{Corollary}

\newtheorem{lem}[thm]{Lemma}

\newtheorem{conj}[thm]{Conjecture}

\newtheorem{expl}[thm]{Example}

\newtheorem{rmrk}[thm]{Remark}

\DeclareMathOperator{\var}{var}

\theoremstyle{example}

\newtheoremstyle{example_contd}
{\topsep} {\topsep}%
{\upshape}% Body font
{}% Indent amount (empty = no indent, \parindent = para indent)
{\bfseries\scshape}% Thm head font
{.}% Punctuation after thm head
{1em}% Space after thm head (\newline = linebreak)
{\thmname{#1} \thmnumber{ #2}\thmnote{#3} (continued)}% Thm head spec

\theoremstyle{example_contd}

\begin{document}

\begin{abstract}
Let $G$ be a graph. In a famous paper Collatz and Sinogowith had proposed to measure its deviation from regularity by the difference of the (adjacency) spectral radius and the average degree: $\epsilon(G)=\rho(G)-\frac{2m}{n}$.

We obtain here a new upper bound on $\epsilon(G)$ which seems to consistently outperform  the best known upper bound to date, due to Nikiforov. The method of proof may also be of independent interest, as we use notions from numerical analysis to re-cast the estimation of $\epsilon(G)$ as a special case of the estimation of the difference between Rayleigh quotients of proximal vectors.
\end{abstract}



\thanks{{This research was supported by the Israel Science Foundation (grant number 862/10.)}}

\maketitle

\section{Introduction and main result}
Let $G$ be a connected graph with adjacency matrix $A$. Then $A$ has a Perron value $\rho$ and a positive Perron unit vector $v$ that satisfy
$$
Av=\rho v, ||v||_{2}=1.
$$
%For a real symmetric matrix $M$ we shall denote its largest and smallest eigenvalues by $\lambda_{\max}(M)$ and $\lambda_{\min}(M)$, respectively, and also let $\spr{M}=\lambda_{\max}(M)-\lambda_{\min}(M)$, the \emph{spread} of the matrix. 
%The spectral radius of a nonnegative matrix $A$ 

Suppose that the graph $G$ has $n$ vertices and $m$ edges. Then $\frac{2m}{n}$ is equal to the average vertex degree of $G$. The following classic result of Collatz and Sinogowitz relates it to the Perron value:
\begin{thm}\cite{ColSin57}\label{thm:cs}
Let $G$ be a connected graph with $n$ vertices and $m$ edges. Then 
\begin{equation}\label{eq:cs}
\rho \geq \frac{2m}{n}
\end{equation}
and equality holds if and only if $G$ is regular.
\end{thm}

Theorem \ref{thm:cs} allows us to consider $\epsilon(G)=\rho-\frac{2m}{n}$ as a measure of the graph's irregularity. As such it has been studied by various authors \cite{Extremal16,Bel92,CioGre07,Nik06,Nik07}. For some alternative ways of measuring the irregularity see \cite{ElpWoc13}. We also call attention to \cite{Nik06walk} where \eqref{eq:cs} is placed in a wider context.

The inspiration for the present paper is given by the results of Nikiforov \cite{Nik06} who related $\epsilon(G)$ to two other natural measures of irregularity which are based on the degree sequence of $G$. These are in fact the first two moments of the degree sequence:
%The two measures he considered are:
$$
s(G)=\sum_{u \in V(G)}{\Big|d_{u}-\frac{2m}{n}\Big|}
$$
and
$$
\var(G)=\frac{1}{n}\sum_{u \in V(G)}{\Big(d_{u}-\frac{2m}{n}\Big)^{2}},
$$
where $d_{u}$ is the degree of the vertex $u$.

As observed in \cite{Nik06}, $s(G)$ and $\var(G)$ are related:
%\begin{equation}\label{eq:var_s_relation}
$$
\frac{s^{2}(G)}{n^{2}} \leq \var(G)\leq s(G).
$$
%\end{equation}

Another interesting property of $\var(G)$ can be obtained from the Popoviciu inequality \cite{Pop35}:
$$
\var(G) \leq \frac{(\Delta(G)-\delta(G))^{2}}{4}.
$$

%Nikiforov has established in \cite{Nik06} a relation between $\epsilon(G)$ and the two aformentioned quantities:

%Nikiforov has proved in 2006 an upper bound on the CS measure, in terms of yet another measure of irregularity. Denote the degree of vertex $v\in V(G)$ by $d_{v}$ and let $s(G)=\sum_{v \in V(G)}{|d_{v}-\frac{2m}{n}|}$. The maximum and minimum degree will be denoted, as usual, by $\Delta(G)$ and $\delta(G)$, respectively. 

Our goal is to improve on the following result of Nikiforov:
\begin{thm}\cite{Nik06}\label{thm:nik}
Let $G$ be a graph. Then
\begin{equation}\label{eq:nik}
\frac{\var(G)}{2\sqrt{2m}} \leq \epsilon(G)\leq \sqrt{s(G)}.
\end{equation}
\end{thm} 

%%From \eqref{eq:var_s_relation} and \eqref{eq:nik} we can obtain the following upper bound as well:
%%$$
%%\epsilon(G) \leq \sqrt[4]{n^{2}\var(G)}.
%%$$

Let $S=||v||_{1}$, that is the sum of the entries of the unit Perron eigenvector. Note that by Cauchy-Schwarz, $S^{2} \leq n$, with equality iff the graph $G$ is regular, and thus $S^{-1}$ may also serve as a measure of irregularity.

\begin{thm}\label{thm:main}
Let $G$ be a connected graph. Then
\begin{equation}\label{eq:main}
\epsilon(G) \leq \sqrt{\var(G)} \cdot \sqrt{\frac{n}{S^{2}}-1}.
\end{equation}
\end{thm}

The proof requires a brief detour into the field of numerical analysis, taking \cite[Section 2]{ZhuArgKny13} as our benevolent guide. Let $M$ be a real symmetric matrix and $x \neq 0$ a (real) vector. The \emph{Rayleigh quotient} is 
$$
\varrho(x)=\frac{x^{T}Mx}{x^{T}x}.
$$
It is well-known that the eigenvalues of $M$ are precisely the stationary points of $\varrho(\cdot): \mathbb{R}^{n} \rightarrow \mathbb{R}$. 

Suppose now that $Ax=\lambda x$, so that $\varrho(x)=\lambda$. Suppose also that $y$ is a vector lying close to $x$. We can expect by the continuity of $\varrho(\cdot)$ that $\varrho(x)$ will be close to $\varrho(y)$. Since the function $\varrho(\cdot)$ is homogenous, the useful way to measure proximity of vectors will be by the angle between $x$ and $y$:
$$
\angle(x,y)=\arccos{\frac{|\left< x,y \right>|}{||x||_{2} \cdot ||y||_{2}}}.
$$

There are two ways of making this statement precise: the \emph{a priori} bound \eqref{eq:apriori} and the \emph{a posteriori} bound \eqref{eq:aposteriori}. The latter bound uses the residual vector $r(y)=Ay-\varrho(y)y$.

\begin{equation}\label{eq:apriori}
|\lambda-\varrho(y)| \leq (\lambda_{\max}(M)-\lambda_{\min}(M)) \cdot \sin^{2}(\angle(x,y)).
\end{equation}

\begin{equation}\label{eq:aposteriori}
|\lambda-\varrho(y)| \leq \frac{||r(y)||}{||y||} \cdot \tan(\angle(x,y)).
\end{equation}

It is not possible to tell in advance which of the bounds will turn out more useful for a particular problem. For our purposes the a posteriori works much better, so we will henceforth focus on it.

Let us take $M=A$ and $x=v$ and $y=\mathbf{1}_{n}$. Then we have that $\lambda=\rho(G)$ and $\varrho(y)=\frac{2m}{n}$.
The residual vector $r(y)$ is:
$$
r(y)=Ay-\varrho(y)y=d-\frac{2m}{n}.
$$
Therefore
\begin{equation}\label{eq:aux1}
\frac{||r(y)||}{||y||}=\sqrt{\var(G)}.
\end{equation}
On the other hand, the cosine of the angle between $v$ and $\mathbf{1}_{n}$ is:
\begin{equation}\label{eq:aux2}
\cos \angle(v,\mathbf{1}_{n})= \frac{|\left< v,\mathbf{1}_{n} \right>|}{||v||_{2} \cdot ||\mathbf{1}_{n}||_{2}}=\frac{S}{\sqrt{n}}.
\end{equation}

The claim of Theorem \ref{thm:main} now follows from (\ref{eq:aposteriori},~\ref{eq:aux1},~\ref{eq:aux2}). \qed

\begin{rmrk}
Extensive numerical calculations indicate that the bound of \eqref{eq:main} is stunningly close to the true value of $\epsilon(G)$ in all cases examined. However, the actual estimation of $S^{2}$ on which the bound depends is often very difficult. Therefore, we are willing to settle for a weaker bound: $\epsilon(G) \leq \sqrt{\var(G)}$ which would still improve upon Theorem \ref{thm:nik}. This fails to be true for disconnected graphs but we strongly believe that it is true for connected graphs.
\end{rmrk}

%The rest of the paper will be devoted to either calculating or estimating the tangens term of \eqref{eq:main} in various cases. Usually it is very small, and we have not been able to find an any example of a connected graph for which it exceeds $1$. In fact we feel rather confident in offering the following
\begin{conj}\label{conj:var}
If $G$ is connected, then
$$
\epsilon(G) \leq \sqrt{\var(G)}.
$$
\end{conj}
Clearly, the conjecture is equivalent to $$S^{2}\geq^{?} \frac{n}{2}.$$
%If true, the conjecture improves upon Theorem \ref{thm:nik} for connected graphs. 

\section{First examples - exact computation of $S^2$}
In order to demonstrate the strength of Theorem \ref{thm:main} we would like to consider fist a number of examples in which the Perron vector $v$ can be easily computed explicitly, and therefore a formula for $S^{2}$ can be written down. 

Later we will develop some ways of estimating $S^{2}$ from below in cases where explicit expressions for $v$ 
are not available or are too intimidating to be effectively used.

%Let us now introduce a useful tool that will be used a number of times: a partition $\pi$ of the vertex set $V(G)$ into cells $\pi_{1},\pi_{2},\ldots,\pi_{m}$ is called \emph{equitable} if any vertex $v \in \pi_{i}$ has $a^{\pi}_{ij}$ neighbours in $\pi_{j}$, irrespective of the choice of $v$. The partition can be described by a matrix: $A_{\pi}=(a^{\pi}_{ij})$. 

%Equitable partitions have long been used in the study of adjacency matrices (cf. \cite[Section 9.3]{AGT2} or \cite[Section 2.4]{Eigenspaces}). Our Lemma \ref{lem:part} is a weaker version of  \cite[Theorem 2.4.6]{Eigenspaces}:
%\begin{lem}\label{lem:part}
%Let $G$ be a graph with equitable partition $\pi$. Then:
%\begin{itemize}
%\item
%Every eigenvalue of $A_{\pi}$ is an eigenvalue of $A$. 
%\item
%The Perron values of $A$ and $A_{\pi}$ are equal. 
%\item
%If $x$ is an eigenvector of $A_{\pi}$ for the eigenvalue $\lambda$, then the vector $\hat{x}$ which takes the value $x(i)$ on all vertices in $\pi_{i}$ is an eigenvector of $A$ for the eigenvalues $\lambda$.
%\end{itemize}
%\end{lem}

%\begin{expl}
\subsection{Bicliques}
Let $G=K_{p,q}$ be a complete bipartite graph, with $p \leq q$. It is not hard to compute that $\rho(G)=\sqrt{pq}$ and $$\epsilon(G)=\sqrt{pq}-\frac{2pq}{q+p}.$$
Nikiforov's estimate is:
$$
\epsilon(G) \leq \sqrt{s(G)}=\sqrt{2pq\frac{(q-p)}{q+p}},
$$
which has the correct order of magnitude but is off by multiplicative and additive constants. Let us now compute the bound of Theorem \ref{thm:main}:
$$
\sqrt{\var(G)}=\frac{(q-p)}{q+p}\sqrt{pq},
$$
and the Perron vector of 
$$
A(G)=\left[
        \begin{array}{cc}
           J_{p}& 0\\
					0&J_{q}
        \end{array}
    \right]
$$
%is easily verified via Lemma \ref{lem:part} to be
is easily verified to be
$$
v=\left[
        \begin{array}{c}
           \frac{1}{\sqrt{2q}} \cdot j_{p}\\
           \frac{1}{\sqrt{2p}} \cdot j_{q}
        \end{array}
    \right].
$$
Therefore $S=\frac{1}{\sqrt{2}}(\sqrt{q}+\sqrt{p})$ and
$$
\sqrt{\frac{n}{S^{2}}-1}=\frac{\sqrt{q}-\sqrt{p}}{\sqrt{q}+\sqrt{p}}.
$$
Finally, a simple algebraic manipulation will show that in this case equality obtains in \eqref{eq:main} and thus our bound is sharp.
%\end{expl}

\subsection{Harmonic graphs}
A graph is called \emph{harmonic} \cite{DreGut03,Gru02} if for some real $\lambda$ the equality $\lambda d_{v_{i}}=\sum_{j \sim i}{d_{v_{j}}}$ holds for all $i=1,2,\ldots,n$. This is clearly equivalent to
$$\rho=\lambda, \quad v=c \cdot d, c>0.$$ In this case we can evaluate the term $S$ precisely.

Let us define the quantity 
$$Z_{G}=\sum_{v \in V(G)}{d_{v}^{2}}$$ 
(cf. e.g. \cite{Nik07pow,GutDas04}). Then 
%$$
\begin{equation}\label{eq:h0}
v=\sqrt{Z_{G}}\cdot d
\end{equation}
%$$ 
for a harmonic graph.
%and we can restate Theorems \ref{thm:mirsky} and \ref{thm:deg} thus:
\begin{thm}
Let $G$ be a harmonic graph on $n$ vertices and with $m$ edges. Then
\begin{equation}\label{eq:harm}
\epsilon(G) \leq \sqrt{\var(G)} \cdot \sqrt{\frac{nZ_{G}}{4m^{2}}-1}.
\end{equation}
\end{thm}
\begin{proof}
By \eqref{eq:h0} we have $$S^{2}=\frac{4m^{2}}{Z_{G}}.$$
\end{proof}

%%\begin{expl}
%%The graph in Figure \ref{fig:1} is $3$-harmonic (cf. \cite[p. 43]{BorGryGutPet03}).
%%
%%\begin{figure}[here]\label{fig:1}
%%\includegraphics[height=4cm,width=5cm]{h3} 
%%\caption{A $3$-harmonic graph}
%%\end{figure}
%%As $n=7$ and $m=10$ we have $\epsilon(G)=3-\frac{20}{7}=\frac{1}{7}\approx 0.1429$. For Nikiforov's bound we have $s(G)=\frac{24}{7}$ and therefore \eqref{eq:nik} gives an estimate of $1.8156$. Even applying \eqref{eq:nikc} we have an estimate of $1.3093$. The naive-type bounds $\Delta-\frac{2m}{n}$ and $\sqrt{2m-n+1}-\frac{2m}{n}$ do somewhat better here, yielding $1.1429$ and $0.8845$ respectively.
%%
%%To obtain the best estimate we use \eqref{eq:m1}: $\epsilon(G) \leq \frac{2\sqrt{10}}{21}\approx 0.3012$.
%%\end{expl}

\begin{expl}
Consider a family of $3$-harmonic graphs, constructed in \cite{BorGryGutPet03}. See Figure \ref{fig:tk}. The graph $T_{k}$ has $n=3k$ vertices and $m=4k$ edges. It has $k$ vertices of degree $4$ and $2k$ vertices of degree $2$, therefore $\frac{2m}{n}=\frac{8}{3}$ and $\epsilon(G)=\frac{1}{3}$.

Nikiforov's estimate \eqref{eq:nik} gives: 
$$\epsilon(G) \leq \sqrt{s(G)}=\sqrt{\frac{8k}{3}},$$
failing to flesh out the fact that $\epsilon(G)$ is constant for the whole family. 

On the other hand, $Z_{G}=24k$ and $\var(G)=\frac{8}{9}$ and thus from \eqref{eq:harm}:
$$
\epsilon(G) \leq \sqrt{\var(G)} \cdot \sqrt{\frac{nZ_{G}}{4m^{2}}-1} = \sqrt{\frac{1}{8}}\cdot\sqrt{\frac{8}{9}}=\frac{1}{3}.
$$
So once again, equality holds.
\end{expl}
 
\begin{figure}[h]
\begin{center}$
\begin{array}{cc}
\includegraphics[width=2.5in]{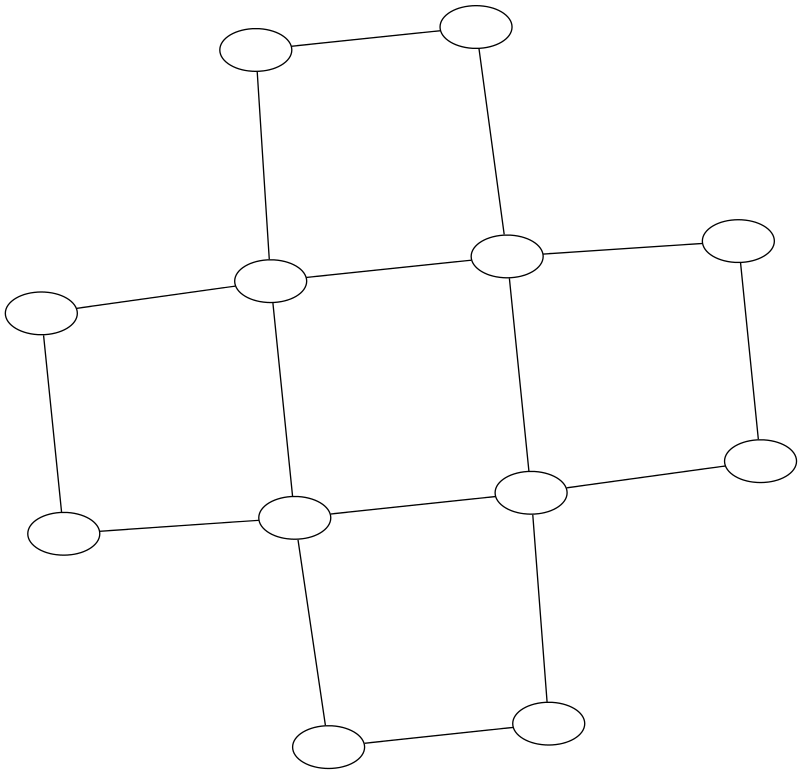} &
\includegraphics[width=2.5in]{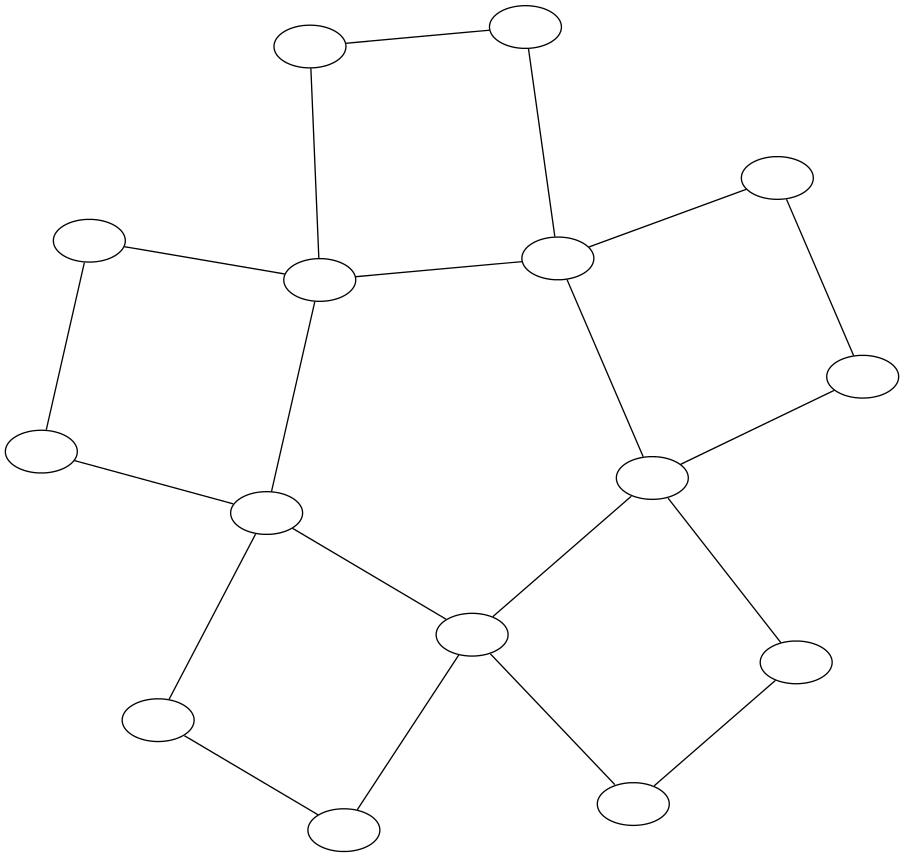}
\end{array}$
\end{center}
\caption{$T_{4}$ and $T_{5}$}\label{fig:tk}
\end{figure}

\section{Estimating $S^{2}$ by the Wilf method}

Recall the classic result due to Wilf:  
\begin{thm}\cite{Wilf86}\label{thm:wilf}
Let $G$ be a graph with clique number $\omega$ and spectral radius $\rho$. Then 
%$$
\begin{equation}\label{eq:wilf}
S^{2} \geq \frac{w}{w-1}\rho.
\end{equation}
%$$
\end{thm}

Wilf's result is in many cases sufficiently powerful to obtain, in conjunction with Theorem \ref{thm:main}, excellent estimates on $\epsilon(G)$. In particular we can use it to prove our Conjecture \ref{conj:var} in a special case:
\begin{cor}\label{cor:w}
If $G$ is a connected graph on $n$ vertices with $\omega(G) \geq \frac{n}{2}$, then
$$
\epsilon(G) \leq \sqrt{\var(G)}.
$$
\end{cor}
\begin{proof}
Since the spectral radius is monotone with respect to subgraphs (cf. \cite[p. 33]{Spectra_BH}) we have $\rho \geq \omega-1$. Therefore $S^{2} \geq \omega \geq \frac{n}{2}$.
\end{proof}

In the remainder of this section we will study a particular example.

\subsection{Pineapples}
The pineapple graph $P(n,q)$ consists of a clique on $q$ vertices and a stable set on $n-q$ vertices, so that one particular vertex in the clique in adjacent to all the vertices in the stable set. Pineapple graphs have high values of $\epsilon(G)$ and in fact have been conjectured to be its maximizers:
\begin{conj}\cite{Extremal16}
Among all graphs on $n \geq 10$ vertices the graph $G$ with the highest value of $\epsilon(G)$ is $G=PA(n,q), q=\lfloor \frac{n}{2} \rfloor +1$.
\end{conj}

\begin{figure}[h]
\begin{center}
\includegraphics[width=2.5in]{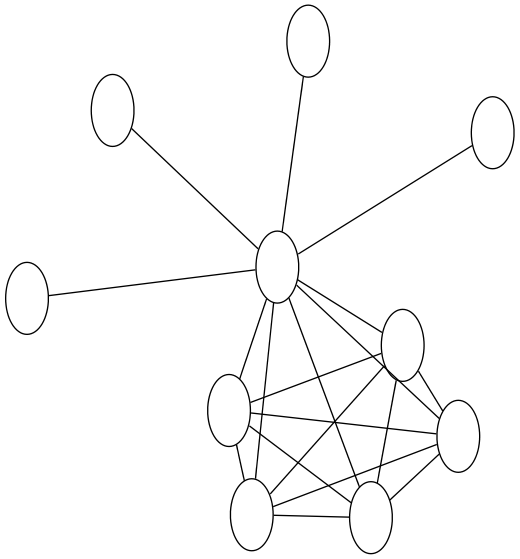} 
\end{center}
\caption{The graph $P(10,6)$}
\end{figure}

\begin{expl}
Consider the graph $G=P(2k,k+1)$. Nikiforov's estimate is:
$$
\epsilon(G) \leq \sqrt{s(G)}=\sqrt{\frac{k^{3}-3k+2}{k}}.
$$
\end{expl}

Since $\omega=k+1 \geq k=\frac{n}{2}$ we can use Corollary \ref{cor:w} to obtain:
$$
\epsilon(G) \leq \sqrt{\var(G)} = \frac{k-1}{2k} \cdot \sqrt{k^2+4k-4}.
$$

Thus an improvement by a factor of two is gained.

%\begin{rmrk}
%As a matter of fact, it is possible to obtain a precise expression for the Perron vector of the former, using Lemma \ref{lem:part} and to look up Cioaba \& Gregory's computation for the Perron vector of the latter. However, such an approach would be extremely laborious, whereas using Wilf's theorem will be short and sweet.
%\end{rmrk}

\section{Estimating $S^{2}$ for cones}\label{sec:dom}

%\section{Graphs with a dominating vertex}\label{sec:dom}
Let us now consider the case when $\Delta(G)=n-1$, \emph{i.e.} when some vertex is adjacent to all other vertices. Such a vertex is called \emph{dominating} or \emph{universal}. Denote by $H$ the subgraph obtained by deleting $v$ and all edges incident upon it from $G$. Another common mode of speaking is to say that $G$ is the \emph{cone} over $H$ and the notation $G=H \vee K_{1}$ is used accordingly.

The pineapple graph is in fact a cone over the disjoint union of a clique and a stable set. As we have seen, for the pineapple graph the Wilf method works very well. 

However, in other cases, it may yield poor results. Therefore we shall now develop an alternative method of estimating $S^{2}$ specifically for cones and then illustrate its power by an example.
%We would now like to exhibit a graph for which the Wilf method of estimating $S^{2}$ yields less satisfactory results.
%, in order to motivte the development later in this section of another method of estimation.

%\begin{example}\label{ex:G10}
%Let $H=K_{3} \cup 2K_{2} \cup 2K_{1}$ and let $G= H \vee K_{1}$. Then $n=10$ and $\omega=4$. The spectral radius of $G$ is $\rho=3.7321$ and so if we were to use here Theorem \ref{thm:wilf} we would have had $S^{2} \geq 4.9761$  - falling short of $\frac{n}{2}$ and thus yielding a final estimate of $\epsilon(G)$ worse than $\sqrt{\var(G)}$. 
%\end{example}

%To remedy this situation, we now introduce another approach to estimating $S^{2}$. 
%Recall that we denote the Perron eigenpair of $G$ by $(\rho,v)$ and assume that $v$ is a unit vector, \emph{i.e.} $||v||_{2}=1$. 

\begin{thm}\cite{Gol14}
Let $G$ be a connected graph on vertices $\{1,2,\ldots,n\}$ with Perron vector $v$, normalized so that $||v||_{2}=1$. For every vertex $i \in V(G)$ let $H_{i}=G-\{i\}$ be the subgraph obtained by deleting $i$ from $G$ and let $\rho_{H_{i}}$ be its spectral radius. Then for any $1 \leq i \leq n$:
\begin{equation}\label{eq:felix}
v_{i}^{2} \geq \frac{1}{1+\frac{d_{i}}{(\rho-\rho_{H_{i}})^{2}}}.
\end{equation}
\end{thm}

\begin{lem}\label{lem:d1}
Let $G$ be a graph with vertex set $\{1,2,\ldots,n\}$ and suppose that $1$ is a  dominating vertex. Then 
\begin{equation}\label{eq:S}
S=(\rho+1)v_{1}.
\end{equation}
\end{lem}
\begin{proof}
Consider the first entries of both sides of the equation $Av=\rho v$:
$$\rho v_{1}=(Av)_{1}=\sum_{j=2}^{n}{v_{j}}=S-v_{1}$$
\end{proof}

Now, combining \eqref{eq:felix} and \eqref{eq:S} we immediately obtain:
\begin{thm}\label{thm:domvert}
Let $G=H \vee K_{1}$ and let $\rho_{H}$ be the spectral radius of $H$. If $G$ has $n$ vertices, then:
\begin{equation}\label{eq:est}
S^{2} \geq \frac{(\rho+1)^{2}(\rho-\rho_{H})^{2}}{(\rho-\rho_{H})^{2}+n-1}.
\end{equation}
\end{thm}

Note that the right-hand side of \eqref{eq:est} is a decreasing function of $\rho-\rho_{H}$ and therefore we can estimate it from below, in turn, by using bounds of the form $\rho \geq a$ and $\rho_{H} \leq b$, to obtain:
$$
S^{2} \geq \frac{(a+1)^{2}(a-b)^{2}}{(a-b)^{2}+n-1}.
$$

\begin{expl}
Let $G=P_{20} \vee K_{1}$ be the cone over the path on $20$ vertices. To fairly compare the bounds on $S^{2}$ provided by \eqref{eq:wilf} and \eqref{eq:est} we will use  Hofmeister's  \cite{Hof88} bound $\rho \geq \sqrt{\frac{1}{n}\sum_{i=1}^{n}{d_{i}^{2}}}$ for both. Since the degrees of $G$ are: $n-1$, $3$ repeated $n-3$ times, and $2$ repeated twice, we have:
$$
\rho \geq \sqrt{\frac{1}{n}(n^2+7n-18)}=5.21.
$$
Thus, \eqref{eq:wilf} yields:
$$
S^{2} \geq \frac{3}{2}\rho \geq 7.815,
$$
whereas \eqref{eq:est} yields, using $\rho_{H} \leq \Delta(H)=2$:
$$
S^{2} \geq 13.115.
$$

The actual value of $S^{2}$ in this example is $16.8305$ while the right-hand side of \eqref{eq:est} is $16.5815$. 

\end{expl}

%%\section{An idea for graphs of diameter two}\label{sec:diam}
%%Recall that the \emph{diameter} of a graph is the largest possible distance between a pair of vertices. The only graphs with diameter one are cliques. Clearly, a non-clique cone is of diameter two (since every two vertices are connected through the dominating vertex). However, this class is much wider that the class of cones and in fact almost every graph has diameter two (cf. \cite[p. 341]{GraphsDigraphs4}).
%%
%%Here we sketch an approach to estimating $S^{2}$ for graphs of diameter two. We do not derive closed-form formulas, leaving this for a further work. However, this appraoch appears to have good potential.
%%
%%Let 
%%$$
%%V(a,b)=\{x \in \mathbb{R}^{n}| \sum_{i=1}^{n}{x_{i}^{2}}=1 \quad \wedge \quad a \leq x_{i} \leq b\}.
%%$$ 

\section{Acknowledgments}
I wish to thank Dr. Clive Elphick for illuminating correspondences on the subject of this paper and Professor Martin Golumbic for a careful reading of a first draft. 
%I am also grateful to Mr. Zeb Brady for suggesting the use of Karamata's theorem in Section \ref{sec:diam}

\bibliographystyle{abbrv}
\bibliography{nuim}

\providecommand{\noopsort}[1]{}
\begin{thebibliography}{10}

\bibitem{Extremal16}
M.~Aouchiche, F.~K. Bell, D.~Cvetkovi\'{c}, P.~Hansen, P.~Rowlinson, S.~K.
  Simi\'{c}, and D.~Stevanovi\'{c}.
\newblock Variable neighborhood search for extremal graphs. {XVI}. {S}ome
  conjectures related to the largest eigenvalue of a graph.
\newblock {\em European J. Oper. Res.}, 191:661--676, 2008.

\bibitem{Bel92}
F.~K. Bell.
\newblock Eigenvalues and degree deviation in graphs.
\newblock {\em Linear Algebra Appl.}, 161:45--54, 1992.

\bibitem{BorGryGutPet03}
B.~Borovi{\'c}anin, S.~Gr{\"u}newald, I.~Gutman, and M.~Petrovi{\'c}.
\newblock Harmonic graphs with small number of cycles.
\newblock {\em Discrete Math.}, 265(1--3):31--44, 2003.

\bibitem{Spectra_BH}
A.~E. Brouwer and W.~H. Haemers.
\newblock {\em Spectra of Graphs}, volume 223 of {\em Universitext}.
\newblock Springer, 2012.

\bibitem{CioGre07}
S.~M. Cioab\u{a} and D.~A. Gregory.
\newblock Large matchings from eigenvalues.
\newblock {\em Linear Algebra Appl.}, 422(1):308--317, 2007.

\bibitem{ColSin57}
L.~Collatz and U.~Sinogowitz.
\newblock Spekter endlicher {G}rafen.
\newblock {\em Abh. Math. Sem. Univ. Hamburg}, 21:63--77, 1957.

\bibitem{DreGut03}
A.~Dress and I.~Gutman.
\newblock The number of walks in a graph.
\newblock {\em Appl. Math. Lett.}, 16:797--801, 2003.

\bibitem{ElpWoc13}
C.~Elphick and P.~Wocjan.
\newblock New measures of graph iregularity.
\newblock \url{http://arxiv.org/abs/1305.3570v4}, 2013.

\bibitem{Gol14}
F.~Goldberg.
\newblock A lower bound on the entries of the principal eigenvector of a graph.
\newblock \url{http://arxiv.org/abs/1403.1479}, 2014.

\bibitem{Gru02}
S.~Gr{\"u}newald.
\newblock Harmonic trees.
\newblock {\em Appl. Math. Lett.}, 15(8):1001--1004, 2002.

\bibitem{GutDas04}
I.~Gutman and K.~C. Das.
\newblock The first {Z}agreb index 30 years after.
\newblock {\em MATCH Commun. Math. Comput. Chem.}, 50:83--92, 2004.

\bibitem{Hof88}
M.~Hofmeister.
\newblock Spectral radius and degree sequence.
\newblock {\em Math. Nachr.}, 139:37--44, 1988.

\bibitem{Nik06}
V.~Nikiforov.
\newblock Eigenvalues and degree deviation in graphs.
\newblock {\em Linear Algebra Appl.}, 414(1):347--360, 2006.

\bibitem{Nik06walk}
V.~Nikiforov.
\newblock Walks and the spectral radius of graphs.
\newblock {\em Linear Algebra Appl.}, 418(1):257--268, 2006.

\bibitem{Nik07}
V.~Nikiforov.
\newblock Bounds on graph eigenvalues {II}.
\newblock {\em Linear Algebra Appl.}, 427(2--3):183--189, 2007.

\bibitem{Nik07pow}
V.~Nikiforov.
\newblock The sum of the squares of degrees: {S}harp asymptotics.
\newblock {\em Discrete Math.}, 307(24):3187--3193, 2007.

\bibitem{Pop35}
T.~Popoviciu.
\newblock Sur les \'{e}quations alg\'{e}briques ayant toutes leurs racines
  r\'{e}elles.
\newblock {\em Mathematica, Cluj}, 9:129--145, 1935.

\bibitem{Wilf86}
H.~S. Wilf.
\newblock Spectral bounds for the clique and independence numbers of graphs.
\newblock {\em J. Comb. Theory, Ser. B}, 40:113--117, 1986.

\bibitem{ZhuArgKny13}
P.~Zhu, M.~E. Argentati, and A.~V. Knyazev.
\newblock Bounds for the {R}ayleigh quotient and the spectrum of self-adjoint
  operators.
\newblock {\em SIAM J. Matrix Anal. Appl.}, 34(1):244--256, 2013.

\end{thebibliography}
\end{document}